\documentclass[9pt,reqno]{amsart} 
\usepackage[pass]{geometry}
\usepackage{lmodern}
\newlength\DX
\paperwidth=\dimexpr\paperwidth-\DX\relax
\hoffset=\dimexpr\hoffset-.5\DX\relax
\newlength\DY
\paperheight=\dimexpr\paperheight-\DY\relax
\voffset=\dimexpr\voffset-.5\DY-.5\footskip\relax
\newtheorem*{theorem*}{Theorem}
\newtheorem*{proposition*}{Proposition}
\usepackage[utf8]{inputenc}
\usepackage[T1]{fontenc}    
\usepackage[colorlinks=true,citecolor=blue,linkcolor=blue]{hyperref}      
\usepackage{url}            
\usepackage{booktabs}       
\usepackage{amsfonts}       
\usepackage{amsthm}
\usepackage{amsmath}
\usepackage{amssymb}
\usepackage{mdframed}
\usepackage{nicefrac}       
\usepackage{microtype}      
\usepackage{mathtools}
\usepackage{xcolor}         
\usepackage{todonotes}
\usepackage{tikz}
\usepackage{tikz-cd}
\usepackage{dsfont} 

\newcommand{\sca}[2]{\langle #1 | #2\rangle}
\newcommand{\nr}[1]{\left\Vert #1\right\Vert}
\newcommand{\abs}[1]{\left\vert #1\right\vert}

\newcommand{\Rsp}{\mathbb{R}}

\newcommand{\Wass}{\mathrm{W}}

\renewcommand{\d}{\mathrm{d}}
\newcommand{\dd}{\mathrm{d}}

\newcommand{\Prob}{\mathcal{P}}

\newcommand*{\T}{\mathcal{T}}

\newcommand{\Ent}{\mathcal{E}}

\newmdtheoremenv{theo}{Theorem}
\newmdtheoremenv{coro}{Corollary}

\newtheorem{theorem}{Theorem}[section]

\newtheorem{lemma}[theorem]{Lemma}

\newtheorem{proposition}[theorem]{Proposition}

\theoremstyle{definition}

\newtheorem{remark}{Remark}[section]

\title[Regularized Moment Measures]{Regularized Moment Measures}
\author{Alex Delalande}
\address{}
\email{delalande.alex@gmail.com}
\author{Sara Farinelli}
\address{Università di Genova, Dipartimento di Matematica, MaLGa, Genova, Italy}

\email{sara.farinelli@edu.unige.it}

\begin{document}

\begin{abstract}

In the work \emph{Dealing with moment measures via entropy and optimal transport}, Santambrogio provided an optimal transport approach to study existence of solutions to the moment measure equation, that is: given $\mu$, find $u$ such that $ (\nabla u)_{\sharp}e^{-u}\mathcal=\mu$. In particular he proves that $u$ satisfies the previous equation if and only if $e^{-u}$ is the minimizer of an entropy and a transport cost. Here we study a modified minimization problem, in which we add a strongly convex regularization depending on a positive $\alpha$ and we link its solutions to a modified moment measure equation $(\nabla u)_{\sharp}e^{-u-\frac{\alpha}{2} \|x\|^2}= \mu$. Exploiting the regularization term, we  study the stability of the minimizers. 
\end{abstract}

\maketitle

\section*{Introduction}
\label{sec:introduction}

A probability measure $\mu$ on $\Rsp^d$ is said to be the \emph{moment measure} of the convex function $u:\Rsp^d\to \Rsp\cup \{+\infty\}$ if $\mu$ is the pushforward through the gradient of $u$ of the log-concave probability measure with density $e^{-u}$, that is
\begin{equation}\label{eq:moment_meas}
 (\nabla u)_{\sharp}e^{-u}\mathcal{L}^d=\mu, 
\end{equation}
where $u$ is called \emph{moment map}. {Despite} stemming from convex geometry, the notion of moment measure has noticeable links with {functional analysis} and optimal transport theory. {The most general result} of well posedness of the problem and of existence and uniqueness is proved in \cite{CORDEROERAUSQUIN20153834} by Cordero-Erasquin and Klartag using techniques linked to functional inequalities.  
Their approach relies on the study of the maximization problem
\begin{align}\label{eq:moment-meas_dual}
\sup_{\varphi}\, \log \int e^{-\varphi^* }\dd x - \sca{\varphi}{\mu},
\end{align}
with $\varphi^{\ast}$ being the Legendre transform of $\varphi$. If $\varphi$ is a maximizer of \eqref{eq:moment-meas_dual}, then $\varphi^\ast$ is a solution of \eqref{eq:moment_meas}.
They also observed that in order to have a meaningful problem, one has to look for convex functions $u:\Rsp^d\to \Rsp\cup\{+\infty\}$, which are \emph{essentially continuous}, that is convex functions that are continuous on the boundary of the set: $\{u<+\infty\}$, $\mathcal{H}^{d-1}$ almost everywhere:  in this case $\mu$ is a moment measure if and only if $\mu$ has barycenter  at the origin and it is not supported on a hyperplane. Moreover, the solution to the moment measure problem is unique up to translation because  of the invariance of the Lebesgue measure with respect to translations.
The same existence results are also obtained in \cite{SANTAMBROGIO2016418} by Santambrogio, via a different variational formulation involving the $2$-Wasserstein distance from optimal transport. Santambrogio's approach relies on the minimization of  the following quantity 
\begin{equation}
\label{eq:moment-measure_wasserstein}
  -W_2^2(\rho, \mu)+\frac{1}{2}M_2(\rho)+\Ent(\rho),
\end{equation}
where $W_2(\rho, \mu)$ is the $2$-Wasserstein distance between $\rho$ and $\mu$,  $M_2(\rho)=\int_{\Rsp^d}\|x\|^2\dd \rho$ is the second moment of $\rho$ and $\Ent(\rho)\coloneqq \int_{\Rsp^d}\rho\log(\rho)\dd x$ is the entropy of $\rho$ with respect to the Lebesgue measure.
The minimization problem above can be restated as 
\begin{equation}
\label{eq:moment-measure_correlation}
    \inf_{\rho \in \Prob_1(\Rsp^d)} \Ent(\rho) + \T(\rho, \mu),
\end{equation}
where $\T(\rho,\mu)$ is the maximal correlation functional between $\rho$ and $\mu$. Then $u$ is a minimizer of \eqref{eq:moment_meas} if and only if the log-concave probability density $e^{-u}$  minimizes \eqref{eq:moment-measure_wasserstein}. Uniqueness of solution relies on the fact that the functional $\rho \mapsto \Ent(\rho) + \T(\rho, \mu)$ is geodesically convex on $\Prob^{a.c}_2(\Rsp^d)$ with respect to the geometry of the $2$-Wasserstein distance on this set. An analogous variational formulation involving optimal transport is also used by Santambrogio and Khanh to tackle the $q$-moment measure problem in \cite{khanh2020q}. See also \cite{kolesnikov2017moment} for  a variation of the moment measure problem studied  by means of the optimal transport formulation.
\\

{Quoting \cite{CORDEROERAUSQUIN20153834} and the references therein, moment measures are closely related to Kähler–Einstein metrics in complex geometry and to the (logarithmic) Minkowski problem in convex geometry. As such, given a moment measure $\mu$, one may be led to the practical problem of finding the corresponding moment map $u$. The numerical resolution of equation~\eqref{eq:moment_meas} can be effectively approached via the variational formulation~\eqref{eq:moment-meas_dual}. Indeed, using for instance the Prékopa–Leindler inequality \cite{prekopa1, leindler, prekopa2}, the functional being maximized in~\eqref{eq:moment-meas_dual} can be shown to be concave, making problem \eqref{eq:moment-meas_dual} a concave maximization problem. In the special case where the measure $\mu = \mu_N$ is finitely supported over a set of $N$ points, the unknown $\varphi = \varphi_N$ in~\eqref{eq:moment-meas_dual} can be parameterized with a finite-dimensional vector in $\Rsp^N$. Therefore in this finite-dimensional setting, problem \eqref{eq:moment-meas_dual} (and thus equation~\eqref{eq:moment_meas}) can be solved numerically using standard first- or second-order convex optimization methods. In the general case where $\mu$ is not finitely supported, problem \eqref{eq:moment-meas_dual} becomes infinite-dimensional and its numerical resolution is more challenging. A natural approach is to approximate $\mu$ with a finitely supported measure $\mu_N$, and then solve the resulting finite-dimensional concave maximization problem. In doing so, one is led to question the quality of approximation of the solutions $\varphi_N$ and $\rho_N$ to problems \eqref{eq:moment-meas_dual} and \eqref{eq:moment-measure_correlation} in terms of the support size $N$. This raises the broader question of the quantitative stability of moment maps with respect to their moment measures. Returning to the variational formulation~\eqref{eq:moment-meas_dual}, one may attempt to address this stability question by leveraging stability versions of the Prékopa–Leindler inequality in order to obtain uniform concavity estimates on the maximized functional. However, existing stability results for the Prékopa–Leindler inequality generally involve large or non-explicit stability exponents \cite{boroczky2021stability, boroczky2023quantitative, figalli2024improvedstabilityversionsprekopaleindler}, which lead to weak uniform concavity estimates for \eqref{eq:moment-meas_dual} and, consequently, to impractical stability bounds for the solutions of the moment measure problem. Motivated by these considerations, we introduce in this note a variation of the moment measure problem that enables the derivation of quantitative stability estimates for its solutions.} \\

{More precisely, we introduce in this note} a strongly convex regularized version of \eqref{eq:moment-measure_correlation}.
Let $\alpha > 0$ be a regularization parameter: we consider the following minimization problem
\begin{equation}
\label{eq:reg-moment-meausre_wasserstein}
    \inf_{\rho \in \Prob_1(\Rsp^d)} \Ent(\rho) + \T(\rho, \mu) + \alpha M_2(\rho).
\end{equation}
After studying existence uniqueness of solutions by following \cite{santambrogio2015optimal} we show  that its solutions, i.e. minimizers, correspond to solutions of a modified moment measure equation, that we call \emph{Gaussian regularized moment measure problem}.
Given $\alpha>0$, we say that $\mu\in \Prob_1(\Rsp^d)$ is an \emph{$\alpha$-Gaussian regularized moment measure} associated to the $\alpha$-Gaussian regularized moment map $u$, if there exists $u:\Rsp^d\to \Rsp$ convex and essentially continuous such that 
\begin{equation}
\nabla u_{\sharp}e^{-u-\frac{\alpha}{2} \|x\|^2}\mathcal{L}^d= \mu.
\end{equation}
One has the following.    
\begin{proposition*} 
 For any $\mu \in \Prob_1(\Rsp^d)$,  problem \eqref{eq:reg-moment-meausre_wasserstein} admits a unique minimizer $\rho \in \Prob_1(\Rsp^d)$, that satisfies $\rho = \frac{1}{Z} e^{-u - \frac{\alpha}{2} \nr{\cdot}^2}$ where $u : \Rsp^d \to \Rsp \cup \{+ \infty \}$ is a convex l.s.c. function such that $\T(\rho, \mu) = \int u \dd \rho + \int u^* \dd \mu$ and $Z$ is a positive normalizing constant.
\end{proposition*}
Notice that in the regularized case, one does not need any condition on the barycenter and on the support of $\mu$ to have the existence of a solution. For this reason, this notion of $\alpha$ Gaussian regularized moment measures, which up to our knowledge 
appears here for the first time, gives an equation for which existence of solution is guaranteed 
in more generality. The Gamma-convergence of the $\alpha$-Gaussian 
moment measure functional to the moment measure functional for $\alpha$ going to $0$, 
that we do not prove here for brevity, follows by standard arguments. Moreover the problem is no more invariant under space translation of $u$.
\\

What motivates our study is the investigation of stability of the moment map with respect to variation of the starting measure. {Up to our knowledge, no results are present in literature in this direction.} The introduction of the regularization term, that makes the problem strongly geodesically convex,
 helps to understand the dependence of the minimizers of \eqref{eq:reg-moment-meausre_wasserstein} with respect to $\mu$. We show the following stability inequality for the minimizers of~\eqref{eq:reg-moment-meausre_wasserstein}: 
\begin{theorem*}
Let $M, B > 0$. Then for any $\mu, \nu \in \Prob_2(\Rsp^d)$ with $M_2(\mu), \, M_2(\nu)\leq M$ and $b(\mu), \,b(\nu)\in B(0,B)$, let $\rho_\mu$ and $\rho_\nu$ be the minimizers of \eqref{eq:reg-moment-meausre_wasserstein}, it holds
\begin{equation}
\Wass_2(\rho_\mu, \rho_\nu) \leq C(d, \alpha, M,B) \Wass_2(\mu, \nu)^{1/2}.
\end{equation}
\end{theorem*}

{As a consequence, in the practical setting where the moment measure $\mu$ is absolutely continuous and approximated by a discrete measure $\mu_N$ supported on $N$ points -- typically satisfying $\Wass_2(\mu, \mu_N) \sim N^{-1/d}$ -- the above theorem guarantees that the solution $\rho_N$ of~\eqref{eq:reg-moment-meausre_wasserstein} associated with $\mu_N$ converges to the solution $\rho$ of~\eqref{eq:reg-moment-meausre_wasserstein} associated with $\mu$ at a comparable rate of $N^{-1/2d}$ in Wasserstein distance.}

We mention for completeness that in a different direction, in \cite{Klartag1} it is proved a continuity result of the moment measure with respect to convergence of the moment map,  while in \cite{fathi2019stein} and \cite{Fathi-Mik} it is proved a stability result of the moment measure with respect to the variations of the moment map in the case where one of the two measures is a Gaussian. 
Finally we mention that
the choice of the regularization $M_2(\rho)$ has been done for simplicity: we believe that by choosing a strongly geodesically convex regularization term regular enough, the conclusions could be analogous: see Remark \ref{remark:different-potential}.

\subsection*{Structure of the paper}
In Section \ref{sec:existence} we study the minimization problem \eqref{eq:reg-moment-meausre_wasserstein}: we show existence, uniqueness of minimizers and characterize them.
In Section \ref{sec:stability} we show a stability result for the minimizers of \eqref{eq:reg-moment-meausre_wasserstein}.

\section{Preliminaries}
\label{sec:preliminaries}
\subsection*{Notations}  For $x,y\in \Rsp^d$, $\sca{x}{y}$ is the usual scalar product, if $f$ is a function and $\mu$ a measure, $\sca{f}{\mu}$ is the  pairing $\int_{\Rsp^d}f\,\dd \mu$. For $v\in \Rsp^d$, $\tau_v:x\mapsto x+v$ is the translation of a vector $v$. $B(0,B)$ is the open ball of $\Rsp^d$ of center $0$ and radius $B$.
\\

$\Prob(\Rsp^d)$ is the space of probability measures on $\Rsp^d$. For any $p \geq 1$ and $\mu \in \Prob(\Rsp^d)$, $M_p(\mu)$ denotes the $p$-th moment of $\mu$, that is $\int \nr{x}^p \dd \mu(x)$, and $\Prob_p(\Rsp^d)$ denotes the subset of probability measures over $\Rsp^d$ with finite $p$-th moment. $\Prob^{a.c}_2(\Rsp^d)$ is the set of probability measures in $\Prob_2(\Rsp^d)$ which are absolutely continuous with respect to the Lebesgue measure. We say that a sequence of probability measures $\mu_n$ narrowly converges to $\mu$ in $\mathcal{P}(\Rsp^d)$, $\mu_n\rightharpoonup\mu$ if $\int f\mu_n\rightarrow\int f \mu$ for any $f\in \mathcal{C}_b(\Rsp^d)$. The barycenter of $\mu\in \Prob(\Rsp^d)$ is $b(\mu)\coloneqq \int_{\Rsp^d}x\, \dd \mu\in \Rsp^d$. We say that a measure is centered if it has barycenter in the origin. For a Borel map $T$ and a measure $\mu$, $T_{\sharp}\mu$ is the pushforward of $\mu$ via $T$. $W_p(\mu,\nu)$ is the $p$-Wasserstein distance.

The \emph{maximum correlation functional} $\T$ is defined for $\rho$ and $\mu$ in  
$ \Prob_1(\Rsp^d)$ as 
$\T(\rho, \mu)$:
\begin{equation*}
\T(\rho, \mu) \coloneqq \max_{\gamma \in \Gamma(\rho, \mu)} \int \sca{x}{y} \dd \gamma(x,y) = \min_{u : \Rsp^d \to \Rsp \cup \{+\infty\} u \text{ convex}} \int u \dd \rho + \int u^* \dd \mu,
\end{equation*}
where $u^{\ast}:\Rsp\to $  is the Legendre transform of $u$ and   $\Gamma(\rho, \mu)$ is the set of probabilities $\gamma\in \mathcal{P}(\Rsp^d\times\Rsp^d )$ such that $(P_1)_{\sharp}\gamma= \rho$ and $(P_2)_{\sharp}\gamma= \mu$ where $P_1$ is the projection on the first component and $P_2$ on the second. The fact that the maximum and minimum are well defined and equality between the maximum and the minimum  follow from classical optimal transport duality arguments (see e.g. \cite{santambrogio2015optimal}). It is worth underlining, for our purposes, that if $u$ is a minimizer, then $\nabla u$ is the optimal transport map between $\mu$ and $\nu$. 
\\

The \emph{entropy functional}, $\Ent$, is defined for $\rho \in \Prob(\Rsp^d)$, as 
\begin{equation*}
\Ent(\rho) \coloneqq 
\begin{cases}
\int \rho(x) \log(\rho(x)) \dd x & \mbox{if } \rho \ll \mathcal{L}^d, \\
+\infty & \text{else}.
\end{cases}
\end{equation*}

We restrict ourselves to the space $(\Prob_2(\Rsp^d),W_2)$. It is known that the entropy functional defined above is convex along $2$-Wasserstein geodesics. Moreover it is lower semicontinuous with respect to sequential narrow convergence of probability measures having equibounded second moments, that is: if $\mu_n\rightharpoonup\mu$ with $M_2(\mu_n)\leq C$ and $\Ent(\mu_n)\leq C$, then $\Ent(\mu)\leq \liminf_{n\to +\infty}\Ent(\mu_n)$, see e.g. \cite[Proposition 2.1]{SANTAMBROGIO2016418}. If $\mu_0$ and $\mu_1$ are in $\Prob^{a.c}_2(\Rsp^d)$, and $t\mapsto \mu_t$ is the $2$-Wasserstein geodesic connecting them, then, $t\mapsto\T(\mu_t,\mu)$ with $\mu\in \Prob_2(\Rsp^d)$, is convex as proven in \cite[Proposition 3.3]{SANTAMBROGIO2016418}. 
\subsection*{Regularization of the optimal transport formulation} Let $\alpha > 0$. We introduce the following regularized version of \eqref{eq:moment-measure_correlation}:
\begin{equation}
\label{eq:reg-moment-meausre_wasserstein2}
    \inf_{\rho \in \Prob_1(\Rsp^d)} \Ent(\rho) + \T(\rho, \mu) + \alpha M_2(\rho).
\end{equation}
We denote 
\begin{align*}
F_\mu : \rho \mapsto \Ent(\rho) + \T(\rho, \mu) + \alpha M_2(\rho).
\end{align*}

Since the second moment functional is strongly geodesically convex over $\Prob_2(\Rsp^d)$, and as stated above, the entropy and the maximal correlation functional are convex along $2$-Wasserstein geodesics, we have that $F_{\mu}$
is strongly geodesically convex over $\Prob^{a.c}_2(\Rsp^d)$: 
we have for any $\rho_0, \rho_1 \in \Prob^{a.c}_2(\Rsp^d)$ connected via the geodesic $(\rho_t)_{0 \leq t \leq 1} \in \Prob^{a.c}_2(\Rsp^d)$ that for any $t \in [0,1]$,
\begin{equation}
\label{eq:strong-geodesic-convexity-F-mu}
    F_\mu(\rho_t) \leq (1-t)F_\mu(\rho_0) + t F_\mu(\rho_1) - \alpha \frac{t (1-t)}{2} \Wass_2^2(\rho_0, \rho_1).
\end{equation}

\subsection*{Gaussian regularized moment measures}
As for the classical case mentioned above, solving \eqref{eq:reg-moment-meausre_wasserstein} is equivalent to solve a \emph{modified moment measure problem}.

Analogously to the classical moment measure problem, given $\alpha>0$, we say that $\mu\in \Prob_1(\Rsp^d)$ is an $\alpha$-Gaussian regularized moment measure associated to the $\alpha$-Gaussian regularized moment map $u$ if there exists $u:\Rsp^d\to \Rsp$ convex and essentially continuous such that 
\begin{equation}\label{eq:alpha-mm}
\nabla u_{\sharp}e^{-u-\frac{\alpha}{2} \|x\|^2}= \mu.
\end{equation}
$u:\Rsp^d\to \Rsp\cup{+\infty}$ convex is essentially continuous if it is  continuous on the boundary of $u<+\infty$, $\mathcal{H}^{d-1}$ a.e. (see     \cite{CORDEROERAUSQUIN20153834}).

Again from classical optimal transport theory, if $u$ convex  satisfies \eqref{eq:alpha-mm} then $\nabla u$ is the optimal transport map for the $L^2$ optimal transport problem between $e^{-u-\frac{\alpha}{2} \|x\|^2}$ and $\mu$.{
It will turn out that any measure $\mu\in \Prob_1(\Rsp^d)$ is an $\alpha$-Gaussian regularized moment measure. }

\section{Existence and characterization of solutions}
\label{sec:existence}
We now state an existence result for problem \eqref{eq:reg-moment-meausre_wasserstein}.

\begin{proposition} \label{prop:existence}
 For any $\mu \in \Prob_1(\Rsp^d)$,  problem \eqref{eq:reg-moment-meausre_wasserstein} admits a unique minimizer $\rho \in \Prob_1(\Rsp^d)$, that satisfies $\rho = \frac{1}{Z} e^{-u - \frac{\alpha}{2} \nr{\cdot}^2}$ where $u : \Rsp^d \to \Rsp \cup \{+ \infty \}$ is a convex l.s.c. function such that $\T(\rho, \mu) = \int u \dd \rho + \int u^* \dd \mu$ and $Z$ is a positive constant.
\end{proposition}

Notice that saying that $\rho = \frac{1}{Z} e^{-u - \frac{\alpha}{2} \nr{\cdot}^2} $ for some $u$ convex is equivalent to saying that $\log(\rho)$ is $\alpha$-strongly convex.
\begin{remark}[Singular targets]
In comparison to \cite{SANTAMBROGIO2016418}, we do not need to assume that $\mu$ is \emph{not} supported on a hyperplane in order to guarantee that the regularized problem \eqref{eq:reg-moment-meausre_wasserstein} admits a solution. Indeed, this is a consequence of the regularization of \eqref{eq:moment-measure_correlation} with the second moment of $\rho$, which allows to easily show the tightness of a minimizing sequence. In comparison, it was shown in \cite[Remark 4.2]{SANTAMBROGIO2016418} that whenever $\mu$ is supported on a hyperplane, problem \eqref{eq:moment-measure_correlation} has value $- \infty$. The construction of \cite[Remark 4.2]{SANTAMBROGIO2016418} cannot be translated to problem \eqref{eq:reg-moment-meausre_wasserstein}, and in this problem, it is actually not possible to find a sequence of measures whose entropy dominates the second moment (more precisely, at fixed zero mean and second moment, the measure that minimizes the entropy is the Gaussian and its entropy is of the order of minus the logarithm of the second moment).
\end{remark}

Before proving Proposition \ref{prop:existence}, we first state a lemma that exhibits the effect of a translation of a measure $\rho$ on the value of $F_\mu(\rho)$. The proof of this lemma is postponed at the end of the section.
\begin{lemma}\label{lemma:translations}
Let $\mu$ and $\rho$ be in $\Prob_1(\Rsp^d)$ and let $\bar \mu\coloneqq (\tau_v)_{\sharp}\mu$, $\bar \rho\coloneqq (\tau_v)_{\sharp}\rho$. Then 
\begin{align}
&F_\mu (\bar{\rho})= F_\mu ({\rho})+ \sca{v}{b(\mu)}+\frac{\alpha}{2}\|v\|^2+\alpha \sca{v}{b(\rho)};\label{eq:relation-changing-rho}\\
&F_{\bar \mu}({\rho})= F_{ \mu}({\rho})+\sca{v}{b(\rho)}.\label{eq:relation-changing-mu}
\end{align}
\end{lemma}
\begin{proposition}\label{prop:translation}
One has 
\begin{equation*}
\min_{\rho\in \Prob_1(\Rsp^d)}F_{\mu}(\rho)= \min_{\rho\in \Prob_1(\Rsp^d), b(\rho)=-\frac{1}{\alpha}b(\mu)}F_{\mu}(\rho),
\end{equation*}
moreover any $\rho$ minimizer of $F_{\mu}$ satisfies 
\begin{equation}\label{eq:bminim}
b(\rho)=-\frac{1}{\alpha}b(\mu).
\end{equation}
In particular if $b(\mu)=0$, then 
\begin{equation*}
\min_{\rho\in \Prob_1(\Rsp^d)}F_{\mu}(\rho)= \min_{\rho\in \Prob_1(\Rsp^d), b(\rho)=0}F_{\mu}(\rho),
\end{equation*}
and any minimizer of $F_{\mu}$ satisfies $b(\rho)=0$.

\end{proposition}
\begin{proof}
\begin{align*}
\min_{\rho\in \Prob_1(\Rsp^d)}F_{\mu}(\rho)= \min_{\tilde\rho\in \Prob_1(\Rsp^d), b(\tilde\rho)=-\frac{1}{\alpha}b(\mu), v=b(\rho)-\frac{1}{\alpha}b(\mu)} F_\mu ({\tilde\rho})+ \sca{v}{b(\mu)}+\frac{\alpha}{2}\|v\|^2+\alpha \sca{v}{b(\tilde\rho)}\\
=\min_{\tilde\rho\in \Prob_1(\Rsp^d), b(\tilde\rho)=-\frac{1}{\alpha}b(\mu), b(\rho)\in \Rsp^d} F_\mu ({\tilde\rho})+\frac{\alpha}{2}\|b(\rho)+\frac{1}{\alpha}b(\mu)\|^2=\min_{\tilde\rho\in \Prob_1(\Rsp^d), b(\tilde\rho)=-\frac{1}{\alpha}b(\mu)} F_\mu ({\tilde\rho}),
\end{align*}
where we used the fact that any $\rho$ can be written as $(\tau_{b(\rho)+\frac{1}{\alpha}b(\mu)})_{\sharp}\tilde \rho$ with $\tilde \rho\in \Prob_1(\Rsp^d)$ and $b(\tilde \rho)=-\frac{1}{\alpha}b(\mu)$.
\end{proof}

\begin{proof}[Proof of Proposition \ref{prop:existence}] We first assume that $b(\mu)=0$. Let $(\rho_n)_{n \geq1} \in \Prob_1^{a.c.}(\Rsp^d)$ be a minimizing sequence in \eqref{eq:reg-moment-meausre_wasserstein}. Since $b(\mu)=0$,  Proposition~\ref{prop:translation} allows us to assume that for all $n \geq 1$, $b(\rho_n) = 0$. We aim first to show the tightness of $(\rho_n)_{n \geq1}$. \\
Proposition 2.1 of \cite{SANTAMBROGIO2016418} guarantees the following lower bound on the entropy functional:
$$ \forall n \geq 1, \quad \Ent(\rho_n) \geq -C- M_1(\rho_n)^{1/2},$$ where $C$ is a computable dimensional constant. 
Similarly, by chosing the product measure as admissible plan between $\rho_n$ and $\mu$,  the following lower bound on the maximal correlation is satisfied: $\forall n \geq 1$,
$$ \T(\rho_n, \mu)\geq \int\sca{x}{y}\d(\rho_n\otimes\mu)(x,y)= \sca{\int x \, \d \rho_n(x)}{\int x\,\d \mu(x)}\geq - M_1(\rho_n) M_1(\mu).$$ 
We thus have for any $n\geq 1$ the lower bound
$$ F_\mu(\rho_n) \geq -C- M_1(\rho_n)^{1/2} - M_1(\mu) M_1(\rho_n) + \frac{\alpha}{2} M_2(\rho_n). $$
Using that $M_2(\rho_n) \geq M_1(\rho_n)^2$ following Jensen's inequality, we thus have for any $n\geq 1$ the lower bound
$$ F_\mu(\rho_n) \geq -C- M_1(\rho_n)^{1/2} - M_1(\mu) M_1(\rho_n) + \frac{\alpha}{2} M_1(\rho_n)^2. $$
This implies that the first moments $M_1(\rho_n)$ are  bounded by a constant independent of $n$ and therefore the sequence $(\rho_n)_{n \geq 1}$ is tight. Therefore we have a narrowly converging subsequence. Up to relabeling, we can assume $\rho_n \rightharpoonup \rho \in \Prob_1(\Rsp^d)$. Let us now show that $\rho$ minimizes \eqref{eq:reg-moment-meausre_wasserstein}.\\
From the fact that $(M_1(\rho_n))_{n \geq 1}$ is a bounded sequence, Proposition 2.1 of \cite{SANTAMBROGIO2016418} entails the lower semicontinuity of the entropy:
$$ \lim \inf_n \Ent(\rho_n) \geq  \Ent(\rho). $$ Similarly, from the fact that for all $n \geq 1$, $b( \rho_n) = 0$, Proposition 3.1  \cite{SANTAMBROGIO2016418} entails the lower semicontinuity of the maximum correlation:
$$ \lim \inf_n \T(\rho_n, \mu) \geq  \T(\rho, \mu). $$
We finally get the lower semicontinuity of the term $\rho \mapsto M_2(\rho)=\sca{\nr{\cdot}^2}{\rho}$ from the fact that $x \mapsto \nr{x}^2$ is a continuous and non-negative function on $\Rsp^d$ (Lemma 1.6 of \cite{santambrogio2015optimal}):
$$ \lim \inf_n \sca{\nr{\cdot}^2}{\rho_n} \geq  \sca{\nr{\cdot}^2}{\rho}. $$ We have then that $\rho\mapsto F_{\mu}(\rho)$ is l.s.c. with respect to the narrow convergence. Thus 
 $\rho$ is a minimizer of \eqref{eq:reg-moment-meausre_wasserstein} and the problem admits a solution.\\

We now describe the minimizers of \eqref{eq:reg-moment-meausre_wasserstein}. Let $\bar{\rho}$ be a minimizer and let $u : \Rsp^d \to \Rsp \cup \{+\infty\}$ be such that $\T(\bar{\rho}, \mu) = \sca{u}{\bar{\rho}} + \sca{u^*}{\mu}$. 
Arguing as in \cite[Theorem 4.1]{SANTAMBROGIO2016418}, we see that $\bar{\rho}$ is concentrated on the set where $$x \mapsto 1 + \log \bar{\rho}(x)  + u(x) + \frac{\alpha}{2} \nr{x}^2 $$ is minimized. In particular it follows that $\bar{\rho} > 0$ on $\{ u < +\infty\}$ (otherwise we would get that there exists $x \in \Rsp^d$ such that $u(x) < +\infty$ and $\bar{\rho}(x) = 0$, which would entail $1 + \log \bar{\rho}(x)  + u(x) + \frac{\alpha}{2} \nr{x}^2 = -\infty$, so that $F_\mu(\bar{\rho}) = -\infty$ which is not possible). Thus denoting $C = \mathrm{essinf}_{x \in \Rsp^d}  1 + \log \bar{\rho}(x)  + u(x) + \frac{\alpha}{2} \nr{x}^2$, we have on $\{ u < +\infty\}$ the formula 
$$ \log \bar{\rho}(x) = C - 1 - u(x) - \frac{\alpha}{2} \nr{x}^2,$$
that is 
\begin{equation}
\label{eq:expression-minimizer}
\bar{\rho}(x) = \frac{1}{Z} e^{-u(x) - \frac{\alpha}{2} \nr{x}^2}.
\end{equation}
This formula is also verified on the set $\{ u = +\infty\}$.\\
Let us finally show the uniqueness of the minimizer of \eqref{eq:reg-moment-meausre_wasserstein}. In expression \eqref{eq:expression-minimizer}, the function $u$ is convex. As such, this function is above an affine function, which allows to show that any minimizer $\bar{\rho}$ of \eqref{eq:reg-moment-meausre_wasserstein} admits finite moments of any order $p \geq 1$. In particular, they have bounded second moments. But the functional $\rho \mapsto \sca{\nr{\cdot}^2}{\rho}$ is strongly geodesically convex on $\Prob_2^{a.c.}(\Rsp^d)$, and so is $F_\mu$: this show the uniqueness of the minimizers of $F_\mu$.

We now remove the assumption that $b(\mu)=0$. Define
\begin{equation*}
\tilde \mu\coloneqq (\tau_{-b(\mu)})_{\sharp}\mu,
\end{equation*}
so that $b(\tilde \mu)=0$, 
 we claim that if $\tilde \rho$ is a minimizer for \eqref{eq:reg-moment-meausre_wasserstein} with datum $\tilde\mu$, then 
\begin{equation*}
\rho(x)\coloneqq \tilde \rho\left(x+\frac{1}{\alpha}b(\mu)\right)=(\tau_{-\frac{1}{\alpha}b(\mu)})_{\sharp}\tilde \rho\,
\end{equation*} 
is a minimizer of \eqref{eq:reg-moment-meausre_wasserstein} with datum $\mu$.
In order to show this we observe that $ \mu\coloneqq (\tau_{b(\mu)})_{\sharp}\tilde\mu$ and $b(\tilde\rho)=0$, so
\begin{align*}
{F_\mu(\rho)}=&{F_{\tilde\mu}(\tilde\rho)}- \frac{\|b(\mu)\|^2}{2\alpha}=\min_{\nu }{F_{\tilde\mu}(\nu)}- \frac{\|b(\mu)\|^2}{2\alpha}= 
\min_{\tilde\nu : b(\tilde \nu)=0}{F_{\tilde\mu}(\tilde \nu)}- \frac{\|b(\mu)\|^2}{2\alpha}\\&= \min_{\nu:b(\nu)=-\frac{1}{\alpha}b(\mu) }{F_\mu(\nu)}=\min_{\nu}{F_\mu(\nu)};
\end{align*}
where the equalities follow by  Proposition \ref{prop:translation}: the second one by using the fact that $\tilde \rho$ is a minimizer. Therefore $\rho= \exp(-u-\frac{\alpha}{2}\|\cdot\|^2)$ with $u$ convex and precisely
\begin{equation}\label{eq:u-trasl}
u(x)=\tilde u(x+\frac{1}{\alpha}b(\mu))+\frac{1}{2\alpha}\|b(\mu)\|^2+\sca{x}{b(\mu)},
\end{equation}
where $\tilde u$ is such that $\tilde \rho= \exp(-\tilde u-\frac{\alpha}{2}\|\cdot\|^2)$. From direct computations it follows that $\T(\rho, \mu) = \int u \dd \rho + \int u^* \dd \mu$.
\end{proof}

\begin{proof}[Proof of Lemma~\ref{lemma:translations}]
For the first part recall that $F_\mu (\bar{\rho}) = \Ent(\bar{\rho}) + \T(\bar{\rho}, \mu) + \frac{\alpha}{2} \sca{\nr{\cdot}^2}{\bar{\rho}}.$ Notice that $\Ent(\bar{\rho}) = \Ent(\rho)$ and that
\begin{align*}
\T(\bar{\rho}, \mu) &= \max_{\bar{\gamma} \in \Gamma(\bar{\rho}, \mu)} \int \sca{x}{y} \dd \bar{\gamma}(x,y) \\
&= \max_{\gamma \in \Gamma(\rho, \mu)} \int \sca{x+v}{y} \dd \gamma(x,y) \\
&= \T(\rho, \mu) + \sca{v}{\int y\, \dd \mu(y) }.
\end{align*}
Finally one has 
\begin{equation*}
\sca{\nr{\cdot}^2}{\bar{\rho}} = \int \nr{x + v}^2 \dd \rho(x) = \sca{\nr{\cdot}^2}{\rho} + \nr{v}^2+\sca{v}{b(\rho)}. \qedhere
\end{equation*}
For the second part use analogously that $\T({\rho}, \bar \mu) = T({\rho},  \mu)+\sca{v}{\int y \, \dd \rho(y) }$.
\end{proof}
\begin{proposition}\label{prop:essential-continuity}
Let $u$ as in Proposition \ref{prop:existence}. Let $\Omega\coloneqq\{u<+\infty\}$. Then $u=+\infty$ $\mathcal{H}^{d-1}$ on $\partial \Omega$. In particular $u$ is essentially continuous. 
\end{proposition}
\begin{proof}
From the proof of Proposition \ref{prop:existence} it is clear that it is enough to prove the Proposition by assuming that $b(\mu)=0$. 
We can proceed by following the steps of the proof of \cite[Theorem 4.3]{SANTAMBROGIO2016418}. Namely, assuming by contradiction that there exists a set of positive $\mathcal{H}^{d-1}$ measure in $\partial \Omega$ where $u<+\infty$ (and so $\rho>0$), we can build $\rho_{\varepsilon}$ competitor in $F_{\mu}$ and $u_{\delta}$ competitor in $\T({\mu},\rho_{\varepsilon})$. We refer the reader to  the proof of \cite[Theorem 4.3]{SANTAMBROGIO2016418} for the definitions of $\rho_{\varepsilon}$ and $u_{\delta}$. Then we have to choose $\delta$ in order to get a contradiction for $\varepsilon$ which goes to zero. By following the computations in the proof of the mentioned Theorem, one has by minimality of $\rho$, 
\begin{equation}\label{eq:competitor}
\Ent(\rho) + \T(\rho, \mu) + \frac{\alpha}{2}M_2(\rho)\leq \Ent(\rho_{\varepsilon}) + \T(\rho_{\varepsilon}, \mu) + \frac{\alpha}{2}M_2(\rho_{\varepsilon}).
\end{equation}
We compute 
\begin{align*}
M_2(\rho_{\varepsilon})=M_2(\rho_{|A_{\varepsilon}^c})+ \frac{1}{2}M_2(\rho_{|A_{\varepsilon}})+ \frac{1}{2}M_2(T_{\sharp}\rho_{|A_{\varepsilon}})
\end{align*}
and 
\begin{align*}
M_2(T_{\sharp}\rho_{|A_{\varepsilon}})=\int_{T(A_{\varepsilon})} |x|^{2}\dd T_{\sharp}\rho_{|A_{\varepsilon}}= \int_{A_{\varepsilon}} |x-\varepsilon e_1|^{2}\dd \rho_{|A_{\varepsilon}}
\end{align*}
so that 
\begin{align}\label{eq:M2}
M_2(\rho_{\varepsilon})= M_2(\rho)-\varepsilon \int_{A_{\varepsilon}}x_1\dd\rho+ \varepsilon^2\rho (A_{\varepsilon}).
\end{align}
Therefore from  \eqref{eq:competitor} and \eqref{eq:M2} we have, where again for the definition of $\chi$ (involved in the definition of $u_{\delta}$ we refer the reader to the proof of the aforementioned paper), 
\begin{equation*}
\rho({A_{\varepsilon}})\ln(2)\leq \frac{1}{2}\rho({A_{\varepsilon}})\delta \chi^{\ast}\left(\frac{\varepsilon e_1}{\delta}\right)+\delta \int \chi \dd \mu+\frac{\alpha}{2}\varepsilon^2\rho({A_{\varepsilon}})-\frac{\alpha}{2}\varepsilon \int_{A_{\varepsilon}}x_1 \dd\rho.
\end{equation*}
We recall that $\rho({A_{\varepsilon}})\geq c_0\varepsilon$  from the fact that $\rho$ is log-concave and by the absurd hypotheses and we chose $\delta=c\varepsilon$ with $c$ such that $c\int \chi \dd \mu<\frac{1}{2}c_0\ln(2)$ so that 
\begin{align*}
&\frac{1}{2}\rho({A_{\varepsilon}})\ln(2)\leq \frac{1}{2}\rho({A_{\varepsilon}})c\varepsilon\chi^{\ast}\left(\frac{\varepsilon e_1}{\delta}\right)+\frac{\alpha}{2}\varepsilon^2\rho({A_{\varepsilon}})-\frac{\alpha}{2}\varepsilon \int_{A_{\varepsilon}}x_1 \dd\rho\\
&\leq \frac{1}{2}\rho({A_{\varepsilon}})c\varepsilon\chi^{\ast}\left(\frac{\varepsilon e_1}{\delta}\right)+\frac{\alpha}{2}\varepsilon^2\rho({A_{\varepsilon}})+\frac{\alpha}{2}\varepsilon \rho({A_{\varepsilon}}) \tilde C,
\end{align*}
where $\tilde C$ is such that $A_{\varepsilon}\subseteq \{(x_1,x'): |x_1|\leq \tilde C, x'\in B\}$, 
 which gives the absurd by letting $\varepsilon$ go to zero.
\end{proof}
\begin{remark}
From the characterization of the unique solution of \eqref{eq:reg-moment-meausre_wasserstein} together with Proposition \ref{prop:essential-continuity} we can deduce that given a minimizer $\rho_{\mu}$ with $\rho_{\mu}=\frac{1}{Z}e^{-u-\frac{\alpha}{2}\|x\|^2}$ with $u:\Rsp^d\to +\infty$ convex and l.s.c. given in Proposition \ref{prop:existence}, then $u$ is the $\alpha$-Gaussian regularized moment map for $\mu$.
Indeed, $u$ has been chosen such that $\T(\bar{\rho}, \mu) = \sca{u}{\bar{\rho}} + \sca{u^*}{\mu}$, so from the classical optimal transport theory one has that $\nabla u$ is the map that pushes forward $\rho_{\mu}$ to $\mu$.
\end{remark}
Notice that any minimizer of problem \eqref{eq:reg-moment-meausre_wasserstein} has finite moments of any order.
\begin{remark}\label{remark:different-potential}
We observe for completeness that if $V : \Rsp^d \to \Rsp$ is a convex \emph{regularization potential}, that is assumed to be $1$-strongly convex, i.e. for any $x, y \in \Rsp^d$ and $t \in [0,1]$ it satisfies
\begin{equation*}
    V((1-t)x + t y ) \leq (1-t) V(x) + t V(y) - \frac{t(1-t)}{2} \nr{x-y}^2,
\end{equation*}
analogous results hold with 
\begin{align*}
F_\mu : \rho \mapsto \Ent(\rho) + \T(\rho, \mu) + \alpha \sca{V}{\rho}.
\end{align*} 
 In particular, any $V$ satisfying $\sca{V}{\rho} \geq c M_1(\rho)^\beta$ for some $c>0$ and $\beta > 1$ should work to guarantee existence of solutions to problem \eqref{eq:reg-moment-meausre_wasserstein}. For simplicity, we handle the case where $V=\frac{1}{2}\|x\|^2$.
\end{remark}

\section{Stability of the minimizers}
\label{sec:stability}
In this section we prove  a stability result for the minimizers of \eqref{eq:reg-moment-meausre_wasserstein}.

For any $\mu \in \Prob_1(\Rsp^d)$, denote $\rho_\mu \in \Prob_1(\Rsp^d)$ the unique minimizer of \eqref{eq:reg-moment-meausre_wasserstein}, that is
\begin{equation*}
    \rho_\mu = \arg \min_{\rho \in \Prob_1(\Rsp^d)}F_{\mu}(\rho)=\arg \min_{\rho \in \Prob_1(\Rsp^d)} \Ent(\rho) + \T(\rho, \mu) + \frac{\alpha}{2} M_2(\rho).
\end{equation*}
Also recall that $b(\mu) = \int_{\Rsp^d} y \dd \mu(y) \in \Rsp^d$ is the barycenter of $\mu$.

\begin{theorem}
\label{prop:stability-nobdd}
Let $M, B > 0$. Then for any $\mu, \nu \in \Prob_2(\Rsp^d)$ with $M_2(\mu), \, M_2(\nu)\leq M$ and $b(\mu), \,b(\nu)\in B(0,B)$,
\begin{equation}\label{eq:wass-rho_mu-rho_nu}
\Wass_2(\rho_\mu, \rho_\nu) \leq C(d, \alpha, M,B) \Wass_2(\mu, \nu)^{1/2}.
\end{equation}
\end{theorem}
\begin{remark}\label{rmk:prox}
From the previous proposition it immediately follows that the map $\mu\mapsto \rho_{\mu}=\arg \min_{\rho} F_{\mu}$ is $\frac{1}{2}$-Hölder continuous on Wasserstein balls.  Normally, given a functional $\mathcal{G}$ defined on the space of probability measures, one refers to the map $\mu\mapsto \arg \min_{\rho} \mathcal{G}(\rho)+\frac{1}{2\tau}W^2_{2}(\rho,\mu)$ for $\tau>0$ as proximity operator of the functional $\mathcal{G}$. For $\alpha= \frac{1}{2}$ our result can be read as $\frac{1}{2}$-Hölderianity of the proximity operator of the Entropy. An analogous $\frac{1}{2}$-Hölderianity result for minimizers of a different class of functionals, that is Wasserstein projections, can be found in \cite[Proposition 2.3.4]{ARC}  (see also \cite[Section 5]{DMSV}). 
\end{remark}

In the proof we use some features of Gaussian probability measures. We recall for completeness that given $\gamma\in \Prob(\Rsp^d)$ a Gaussian probability measure which is isotropic (i.e.\ its covariance matrix is a multiple of the identity) and centered at the origin, its entropy, by a direct computation, can be explicitly expressed in terms of its second moment: 
\begin{equation}
\Ent(\gamma)= -\frac{d}{2}\log\left(\frac{2}{d}\pi e M_2(\gamma)\right).
\end{equation}
In addition the following Lemma holds. 
\begin{lemma}\label{lemma:max_entropy}
For any $M>0$,
\begin{align*}
\min_{\rho \in \Prob_2(\Rsp^d), \, b(\rho)=0, \, M_2(\rho)=M}\Ent(\rho)= -\frac{d}{2}\log\left(\frac{2}{d}\pi e M\right).
\end{align*}
\end{lemma}
\begin{proof}
Let $\gamma_{0,M}$ be the Gaussian isotropic probability density with zero mean and such that $M_2(\gamma_{0,M})=M$, i.e. $\gamma_{0,M}(x)\coloneqq \frac{1}{(2\pi M/d)^{\frac{d}{2}}}e^{-\frac{d\|x\|^2}{2M}}$. We show that for any $\rho \in \Prob_2(\Rsp^d)$ with  $ b(\rho)=0, \, M_2(\rho)=M$, $\Ent(\rho)\geq \Ent(\gamma_{0,M})$. We can assume $\rho$ to be absolutely continuous with respect to the Lebesgue measure. By writing  
\begin{align}\label{eq:entr_kl}
\Ent(\rho)-\int \rho\log(\gamma_{0,M})\dd x = \int \rho \log\left (\frac{\rho}{\gamma_{0,M}}\right)\dd x \geq 0, 
\end{align}
where the inequality follows from the fact that 
$\int \rho \log\left (\frac{\rho}{\gamma_{0,M}}\right)\dd x= D_{KL}(\rho|\gamma_{0,M})$ is the Kullbach Leibler divergence which is known to be non negative (see e.g.\ \cite{Kullback1951OnIA}). Therefore from \eqref{eq:entr_kl} one has 
\begin{align*}
\Ent(\rho)\geq \int \rho\log(\gamma_{0,M})\dd x= \int \rho\log\left(\frac{1}{(2\pi M/d)^{\frac{d}{2}}}e^{-\frac{d\|x\|^2}{2M}}\right)\dd x\\= -\frac{d}{2M}\int\rho  \|x\|^2\dd x-\frac{d}{2}\log \left(\frac{2}{d}\pi M \right), 
\end{align*}
and the conclusion follows by recalling that $M_2(\rho)=M$.
\end{proof}
\begin{proof}[Proof of Theorem \ref{prop:stability-nobdd}.]
Let $\mu, \nu$ as in the hypotheses. Denote $(\rho_t)_{0 \leq t \leq 1} \in \Prob_2^{a.c.}(\Rsp^d)$ the unique $2$-Wasserstein geodesic connecting $\rho_\mu$ and $\rho_\nu$. We leverage the definitions of $\rho_\mu, \rho_\nu$ and the strong geodesic convexity of $F_\mu$ and $F_\nu$ (see \eqref{eq:strong-geodesic-convexity-F-mu}) to write:
    \begin{align*}
         F_\mu(\rho_0) = F_\mu(\rho_\mu) \leq F_\mu(\rho_{1/2}) &\leq \frac{1}{2} F_\mu(\rho_0) + \frac{1}{2} F_\mu(\rho_1) - \frac{\alpha}{8} \Wass_2^2(\rho_0, \rho_1), \\
         F_\nu(\rho_1) = F_\nu(\rho_\nu) \leq F_\nu(\rho_{1/2}) &\leq \frac{1}{2} F_\nu(\rho_0) + \frac{1}{2} F_\nu(\rho_1) - \frac{\alpha}{8} \Wass_2^2(\rho_0, \rho_1).
    \end{align*}
    Combining these two inequalities, using that $\rho_0 = \rho_\mu$ and $\rho_1 = \rho_\nu$ and rearranging yields:
    \begin{equation}
    \label{eq:result-strong-convexity}
        \frac{\alpha}{2} \Wass_2^2(\rho_\mu, \rho_\nu) \leq F_\mu(\rho_\nu) - F_\nu(\rho_\nu) + F_\nu(\rho_\mu) - F_\mu(\rho_\mu).
    \end{equation}
    But 
    \begin{align}
    \label{eq:diff-Fmu-Fnu}
        F_\mu(\rho_\nu) - F_\nu(\rho_\nu) &= \T(\rho_\nu, \mu) - \T(\rho_\nu, \nu) \notag
    \end{align}
    and we recall that 
    \begin{equation*}
    \T(\rho_\nu, \mu)= -W_2^2(\rho_\nu, \mu) + M_2(\rho_\nu) + M_2(\mu).
    \end{equation*}
    By using the two previous observations, the right hand side of \eqref{eq:result-strong-convexity} becomes 
    \begin{align*}
    &W_2^2(\rho_\nu, \nu)-W_2^2(\rho_\nu, \mu)+W_2^2(\rho_\mu, \mu)-W_2^2(\rho_\mu, \nu)\\
    &\leq ( W_2(\rho_\nu, \nu)+W_2(\rho_\nu, \mu))W_2(\mu, \nu)+ ( W_2(\rho_\mu, \mu)+W_2(\rho_\mu, \nu))W_2(\mu, \nu), 
    \end{align*}
    by using the triangular inequality. 
    So it remains to bound 
    \begin{equation}\label{eq:sum-wasser}
    W_2(\rho_\nu, \nu)+W_2(\rho_\nu, \mu)+W_2(\rho_\mu, \mu)+W_2(\rho_\mu, \nu).
    \end{equation}
    We proceed estimating $\Wass_2^2(\rho_\mu, \mu)$. 
    We know that 
    \begin{equation}\label{eq:wass_sum}
     \Wass_2^2(\rho_\mu, \mu) \leq M_2(\rho_\mu) + M_2(\mu) - 2 \sca{b(\rho_\mu)}{b(\mu)}=M_2(\rho_\mu) + M_2(\mu)+ \frac{2}{\alpha}\|b(\mu)\|^2.
    \end{equation}
    where we used \eqref{eq:bminim}. We claim that 
    \begin{equation}\label{eq:bound-2ndmoment-rho}
    M_2(\rho_\mu)\leq C(d,\alpha,M,B).
    \end{equation}
    From the definition of $\rho_{\mu}$ we have that $F_{\mu}(\rho_{\mu})\leq F_{\mu}(\gamma_\mu)$ where $\gamma_{\mu}$ is an isotropic Gaussian probability density on $\Rsp^d$ centered in the origin with $M_2(\gamma_{\mu})=M_2(\mu)+d$, which yields 
    \begin{align}\label{eq:bound_second_moment_partial}
    \frac{\alpha}{2} M_2(\rho_{\mu})+\Ent(\rho_\mu)\leq \Ent(\gamma_\mu)+\T(\gamma_{\mu}, \mu)+\frac{\alpha}{2} M_2(\gamma_{\mu}) -\T(\rho_{\mu}, \mu).
    \end{align}
    We proceed here bounding term by term. 
    For the maximal correlation one has 
    \begin{equation*}
    \T(\gamma_{\mu}, \mu)\leq \sqrt{M_2(\gamma_{\mu})M_2(\mu)},
    \end{equation*}
    and 
    \begin{equation*}
    \T(\rho_{\mu}, \mu)\geq b(\rho_\mu)\cdot b(\mu)= -\frac{1}{\alpha}\|b(\mu)\|^2,
    \end{equation*}
    again by using \eqref{eq:bminim}.
    For the first entropy term we have that 
    \begin{equation}\label{eq:entropy_gauss}
    \Ent(\gamma_\mu)= -\frac{d}{2}\log(\frac{2}{d}e\pi M_2(\gamma_{\mu})),
    \end{equation}
    for the second one we use the  entropy inequality of Lemma \ref{lemma:max_entropy} which gives 
    \begin{equation}\label{eq:max_entropy}
    \Ent(\rho_\mu)\geq \Ent(\gamma_{\rho_\mu})= -\frac{d}{2}\log(\frac{2}{d}e \pi M_2(\rho_\mu)),
    \end{equation}
    where $\gamma_{\rho_\mu}$ is an isotropic Gaussian probability distribution centered at zero, with second moment equal to $M_2(\rho_{\mu})$.
    Summing up, reordering the terms in \eqref{eq:bound_second_moment_partial} and using 
    \eqref{eq:entropy_gauss} and \eqref{eq:max_entropy}, one has 
\begin{align*}
&\frac{\alpha}{2} M_2(\rho_{\mu})- \frac{d}{2}\log(\frac{2}{d}e\pi M_2(\rho_{\mu}))= \frac{\alpha}{2} M_2(\rho_{\mu})+\Ent(\gamma_{\rho_\mu})\leq\frac{\alpha}{2} M_2(\rho_{\mu})+ \Ent(\rho_\mu)\\
&\leq \Ent(\gamma_\mu)+ \frac{\alpha}{2} M_2(\gamma_{\mu})+\T(\gamma_{\mu}, \mu)-\T(\rho_{\mu}, \mu) \\& 
\leq -\frac{d}{2}\log\left(\frac{2}{d}e\pi M_2(\mu)+d\right)+ \frac{\alpha}{2}(M_2({\mu})+d)+M_2(\mu)+d+\frac{1}{\alpha}\|b(\mu)\|^2,
\end{align*}
that gives 
\begin{align*}
M_2(\rho_{\mu})-{\frac{d}{\alpha}}\log( M_2(\rho_{\mu}))\leq-{\frac{d}{\alpha}}\log(M_2({\mu})+d)+ \left(\frac{2}{\alpha}+1\right)(M+d)+\frac{2}{\alpha^2}B^2\\\leq -{\frac{d}{\alpha}}\log(d) +\left(\frac{2}{\alpha}+1\right)(M+d)+\frac{2}{\alpha^2}B^2 \eqqcolon \bar C(d, \alpha, M, B).
\end{align*}
Now there exists $A(\frac{\alpha}{d})$ such that for $M_2(\rho_\mu)>A(\frac{\alpha}{d})$, ${\frac{d}{\alpha}}\log( M_2(\rho_{\mu}))\leq \frac{M_2(\rho_{\mu})}{2}$. Therefore \eqref{eq:bound-2ndmoment-rho} follows with $C(d,\alpha,M,B)= \max\{A(\frac{\alpha}{d}),\bar C(d, \alpha, M, B)\}$.
So by using this bound in \eqref{eq:wass_sum}, one has that 
\begin{equation*}
\Wass_2^2(\rho_\mu, \mu)\leq C(d,\alpha,M,B), 
\end{equation*}
where the constant  $C$ is intended up to relabeling. 
The other terms in \eqref{eq:sum-wasser} can be treated analogously. This completes the proof of \eqref{eq:wass-rho_mu-rho_nu}. 
 \end{proof}
 \begin{remark}
{From the previous proof one can see that the constant  $C(d, \alpha, M,B)$ can be explicitly computed and  degenerates for $\alpha \to 0$.}
 \end{remark}

\section{Regularized moment measures are minimizers}
\label{sec:sufficient}
In Section \ref{sec:existence} we showed that Problem \eqref{eq:reg-moment-meausre_wasserstein} has a unique solution and that its solution is a log concave probability density of the form $\frac{1}{Z} e^{-u - \frac{\alpha}{2} \nr{\cdot}^2}$, where $u$ is a solution to Problem \eqref{eq:alpha-mm}. 
We prove here that every log-concave density of the type $\rho = e^{-u - \frac{\alpha}{2} \nr{\cdot}^2}$ with  $u : \Rsp^d \to \Rsp \cup \{+ \infty \}$ essentially continuous convex function is a solution of  \eqref{eq:reg-moment-meausre_wasserstein} for a suitable measure $\mu$. 

The computations closely follow those in \cite{SANTAMBROGIO2016418}.
We first observe that by following the proof of Proposition 3.1, point $(5)$ of \cite{SANTAMBROGIO2016418}, one can deduce the following Lemma.  
\begin{lemma}\label{lemma:approx}
 Let $\rho$ and $\mu$ be in $\Prob_1(\Rsp^d)$, $b(\mu)=0$. Then there exists a sequence of compactly supported probability measures $(\rho_n)_n$ weakly converging to $\rho$, such that $F_{\mu}(\rho_n)\to F_{\mu}(\rho)$.    
\end{lemma}
\begin{proposition} \label{prop:suff-centered}
 Let $u : \Rsp^d \to \Rsp \cup \{+ \infty \}$ essentially continuous convex function, consider $\bar \rho=  e^{-u - \frac{\alpha}{2} \nr{\cdot}^2}$ and suppose $(\nabla u)_\# \bar \rho = \mu$, then $\bar \rho\in  \Prob_1(\Rsp^d)$  and it solves \eqref{eq:reg-moment-meausre_wasserstein} with datum $\mu$.
\end{proposition}
From the previous proposition we can deduce that  given $\mu\in\Prob_1(\Rsp^d)$, one can find the solution of problem  \eqref{eq:alpha-mm} by finding the solution of \eqref{eq:reg-moment-meausre_wasserstein} and vice versa. 
\begin{proof}[Proof of Proposition \ref{prop:suff-centered}]
We first assume that $b(\mu)=0$. 
We also assume that $\rho$ is in $\Prob_2(\Rsp^d)$ with compact support.  We consider the unique $W_2$-geodesic between $\bar \rho$ and $\rho$: $t\mapsto \rho_t= ((1-t)Id+t T)_\# \bar \rho$ where $T=\nabla v$ is the Brenier map between $\bar \rho $ and $\rho$. 
Since  $\rho \mapsto \Ent(\rho)$ and $\rho \mapsto \T(\rho, \mu)$ and $\rho \mapsto M_2(\rho)$ are geodesically convex in $\Prob_2(\Rsp^d)$, inequality 
\begin{equation*}
\Ent(\rho)+\T(\rho, \mu)+\frac{\alpha}{2}M_2(\rho)\geq \Ent(\bar\rho)+\T(\bar \rho, \mu)+\frac{\alpha}{2} M_2(\bar \rho), 
\end{equation*}
follows by showing that the derivative in $t$ at $t=0$ of 
\begin{equation*}
\Ent(\rho_t)+\T(\rho_t, \mu)+\frac{\alpha}{2}M_2(\rho_t), 
\end{equation*}
is non negative. 
We observe that 
\begin{equation*}
M_2(\rho_t)= \int \abs{x}^{2}d \rho_t= \int \abs{x}^{2}\,d ((1-t)Id+t T)_\# \bar \rho= \int \abs{(1-t)x+t T(x)}^{2}\, d \bar \rho 
\end{equation*}
that differentiating and evaluating at zero gives 
\begin{equation*}
\frac{d}{dt}M_2(\rho_t)|_{t=0}= - 2 M_2(\bar \rho)+ 2 \int x \cdot T(x)\, d \bar \rho. 
\end{equation*}
Now we exploit that 
\begin{equation*}
\frac{d}{dt}\left(\Ent(\rho_t)+\T(\rho_t, \mu)\right)|_{t=0}\geq -\int (\Delta^{ac}v-d)\, d \bar \rho + \int (T-x)\cdot \nabla u \, d \bar \rho, 
\end{equation*}
which can be obtained by combining the results in Proposition 2.1 and Proposition 3.3 of \cite{SANTAMBROGIO2016418} where the derivatives are computed, together with the fact (which follows from Lemma 5 in \cite{CORDEROERAUSQUIN20153834}) that 
\begin{equation*}
\int d e^{-u - \frac{\alpha}{2} \nr{\cdot}^2}\, dx - \int x\cdot (\nabla u+ \alpha x) e^{-u - \frac{\alpha}{2} \nr{\cdot}^2}\, d x \geq 0, 
\end{equation*}
to get that 
\begin{align*}
&\frac{d}{dt}\left(\Ent(\rho_t)+\T(\rho_t, \mu)+\frac{\alpha}{2}M_2(\rho_t)\right)|_{t=0}\\
&\geq -\int (\Delta^{ac}v-d)\, d \bar \rho + \int (T-x)\cdot \nabla u \, d \bar \rho- \alpha M_2(\bar \rho)+ \alpha \int x \cdot T(x)\, d \bar \rho\\
&\geq -\int \Delta^{ac}v\, d \bar \rho + \int T\cdot \nabla u \, d \bar \rho +\alpha \int x \cdot T(x)\, d \bar \rho. 
\end{align*}
Now we observe that 
\begin{equation*}
 -\int_{B(0,R)} \Delta v\, d \bar \rho= -\int_{B(0,R)}T\cdot (\nabla u+ \alpha x)\, d \bar \rho -\int_{\partial B(0,R)}T(x)\cdot n \bar \rho \, d \mathcal{H}^{n-1}
\end{equation*}
which, by letting $R\to +\infty $ gives 
\begin{equation*}
-\int \Delta v\, d \bar \rho= -\int T\cdot (\nabla u+ \alpha x)\, d \bar \rho 
\end{equation*}
and so 
\begin{align*}
&\frac{d}{dt}\left(\Ent(\rho_t)+\T(\rho_t, \mu)+\frac{\alpha}{2}M_2(\rho_t)\right)|_{t=0}\\
&\geq  -\int T\cdot (\nabla u+ \alpha x)\, d \bar \rho +  \int T\cdot \nabla u \, d \bar \rho +\alpha \int x \cdot T(x)\, d \bar \rho\geq 0. 
\end{align*}
The fact that $\bar \rho$ is a minimizer also in the whole class of measures in $\rho\in\Prob_1(\Rsp^d)$ follows from Lemma \ref{lemma:approx}. Therefore the case with $b(\mu)=0$ is proved. Now we remove this hypothesis. Define 
\begin{align*}
\tilde \mu \coloneqq (\tau_{-b(\mu)})_{\sharp}\mu, \quad \quad \tilde \rho\coloneqq (\tau_{\frac{1}{\alpha}b(\mu)})_{\sharp}\rho,
\end{align*}
then $b(\tilde \mu)= 0$ and $\tilde \rho = e^{-\tilde u-\frac{\alpha}{2}\|x\|^2}$, with $\tilde u $ convex (see \eqref{eq:u-trasl}) and $\nabla\tilde u_{\sharp}\tilde \rho= \tilde \mu$ indeed by observing that $\nabla \tilde u= \nabla u \circ\tau_{-\frac{1}{\alpha}b(\mu)}-b(\mu) $, one has 
\begin{align*}
&(\nabla \tilde u)_{\sharp}\tilde \rho=\left(\nabla u \circ\tau_{-\frac{1}{\alpha}b(\mu)}-b(\mu)\right)_{\sharp}\left(\tau_{\frac{1}{\alpha}b(\mu)_{\sharp}}\rho\mathcal{L}^d\right)  \\
&=\left(\nabla u -b(\mu)\right)_{\sharp}\rho\mathcal{L}^d= (\tau_{-b(\mu)})_{\sharp}\mu=\tilde \mu.
\end{align*}
From Proposition \ref{prop:suff-centered}, we know that $\tilde \rho$ is minimizer of $F_{\tilde \mu}$. Now as shown in the proof of Proposition \ref{prop:existence}, if $\tilde \rho$ is a minimizer of $F_{\tilde \mu}$, then $\bar \rho = (\tau_{-\frac{1}{\alpha}b(\mu)})_{\sharp}\tilde{ \rho}$ is a minimizer for $F_{\mu}$ with  $\mu=(\tau_{b(\mu)})_{\sharp}\tilde\mu$.
\end{proof}

\subsection*{Aknowledgements}The authors wish to thank Quentin Merigot and Filippo Santambrogio for the suggestion of the problem and discussions on the subject. The project was mainly developed while the authors were at the Lagrange Mathematics and Computing Research Center, which they wish to thank. 
The second author thanks Simone Di Marino for discussions on the topic.
S.F. benefits from the support of MUR (PRIN project 202244A7YL) and thanks DipE awarded to the DIMA-Unige (CUP
D33C23001110001).  S.F. is a member of INdAM–GNAMPA.\\

\subsection*{Statements and Declarations} The authors have no relevant financial or non-financial interests to disclose.

\bibliographystyle{plain}
\bibliography{ref}

\end{document}